\documentclass[12pt,reqno]{amsart}
\usepackage{amsmath}
\usepackage{amssymb}
\usepackage{amstext}
\usepackage{a4wide}
\usepackage{graphicx}
\allowdisplaybreaks \numberwithin{equation}{section}
\usepackage{color}

\numberwithin{equation}{section}

\newtheorem{theorem}{Theorem}[section]
\newtheorem{proposition}[theorem]{Proposition}
\newtheorem{corollary}[theorem]{Corollary}
\newtheorem{lemma}[theorem]{Lemma}
\newtheorem*{theoremA}{Theorem A}
\theoremstyle{definition}

\newtheorem{definition}[theorem]{Definition}

\theoremstyle{remark}
\newtheorem{remark}[theorem]{Remark}

\begin{document}

\title
{Rotating vortex patches for the planar Euler equations in a disk}

 \author{Daomin Cao, Jie Wan, Guodong Wang, Weicheng Zhan}

\address{Institute of Applied Mathematics, Chinese Academy of Science, Beijing 100190, and University of Chinese Academy of Sciences, Beijing 100049,  P.R. China}
\email{dmcao@amt.ac.cn}
\address{Institute of Applied Mathematics, Chinese Academy of Science, Beijing 100190, and University of Chinese Academy of Sciences, Beijing 100049,  P.R. China}
\email{wanjie15@mails.ucas.edu.cn}
\address{Institute for Advanced Study in Mathematics, Harbin Institute of Technology, Harbin {\rm150001}, P. R. China}
\email{wangguodong14@mails.ucas.ac.cn}
\address{Institute of Applied Mathematics, Chinese Academy of Science, Beijing 100190, and University of Chinese Academy of Sciences, Beijing 100049,  P.R. China}
\email{zhanweicheng16@mails.ucas.ac.cn}

%\thanks{This work is partially supported by ARC}

\begin{abstract}
We construct a family of rotating vortex patches with fixed angular velocity for the two-dimensional Euler equations in a disk. As the vorticity strength goes to infinity, the limit of these rotating vortex patches is a rotating point vortex whose motion is described by the Kirchhoff-Routh equation. The construction is performed by solving a variational problem for the vorticity which is based on an adaption of Arnold's variational principle. We also prove nonlinear orbital stability of the set of maximizers in the variational problem under $L^p$ perturbation when $p\in[{3}/{2},+\infty)$.
\end{abstract}

\maketitle

\section{Introduction}
The motion of an ideal fluid of unit density in the plane is governed by the well-known Euler equations
\begin{equation}\label{1-1}
\begin{cases}
\partial_t\mathbf{v}(x,t)+(\mathbf{v}\cdot\nabla)\mathbf{v}(x,t)=-\nabla P(x,t),\,\,x=(x_1,x_2)\in \mathbb R^2, t>0,\\
 \nabla\cdot\mathbf{v}(x,t)=0,
\end{cases}
\end{equation}
where $\mathbf{v}=(v_1,v_2)$ is the velocity field and $P$ is the scalar pressure.
By introducing the scalar vorticity $\omega=curl\mathbf{v}:=\partial_1v_2-\partial_2v_1$ and applying the Biot-Savart law, we get the following vorticity form of \eqref{1-1}(see \cite{MB} or \cite{MP4})
\begin{equation}\label{1-2}
\begin{cases}
\partial_t\omega+\mathbf{v}\cdot\nabla\omega=0,\\
\mathbf{v}(x,t)=\omega*\frac{1}{2\pi}\frac{-x^\perp}{|x|^2}:=\int_{\mathbb R^2}-\frac{1}{2\pi}\frac{(x-y)^\perp}{|x-y|^2}\omega(y,t)dy,
\end{cases}
\end{equation}
where $x^\perp:=(x_2,-x_1)$ denotes clockwise rotation through ${\pi}/{2}$. The vorticity equation \eqref{1-2} means that the vorticity $\omega$ is transported by $\mathbf{v}$, a velocity field determined by $\omega$ itself via the Biot-Savart law.

The famous result of Yudovich asserts that for any initial vorticity $\omega_0\in L^1\cap L^\infty(\mathbb R^2)$, there is a unique weak solution $\omega\in L^\infty((0,+\infty);L^1\cap L^\infty(\mathbb R^2))$ to \eqref{1-2}. An important type of weak solutions appropriate for modeling an isolated region of vorticity with discontinuity is the vortex patch solution, that is, the initial vorticity has the form
\begin{equation}
\omega_0(x)=\lambda I_{A_0}:=\begin{cases}\lambda,&x\in A_0,\\
0,&x\notin A_0,\end{cases}
\end{equation}
where $\lambda\in \mathbb R$ is a parameter representing the vorticity strength.
Since the vorticity is transported by the divergence-free velocity $\mathbf{v}$, we conclude that the evolved solution $\omega(x,t)$ still has the form $\omega(x,t)=\lambda I_{A_t}$ with $|A_t|=|A_0|$, where $|\cdot|$ denotes the two-dimensional Lebesgue measure. A very special example is when $A_0$ is a disk. In this case, it is easy to check that $A_t=A_0$ for all $t>0$. Another remarkable example discovered by Kirchhoff is that $A_0$ is an ellipse centered at the origin with semi-axis $a$ and $b$. In this situation, it can be proved that
$A_t$ is given by
  \begin{equation}
 A_t=e^{i\Omega t}A_0:=\{e^{i\Omega t}x\mid x\in A_0\},
\end{equation}
where
\begin{equation}
e^{i\Omega t}x:=(|x|\cos(\theta_x+\Omega t),|x|\sin(\theta_x+\Omega t))\,\,\text{for each }x=(|x|\cos\theta_x,|x|\sin\theta_x),
\end{equation}
and $\Omega\in\mathbb R$ is the angular velocity determined by $\Omega=(\lambda ab)/(a+b)^2$. See \cite{MB}, Chapter 8.

An interesting question is that is there any other type of rotating vortex patches in the plane?  There are many works in this respect. Here we list some of the relevant and significant ones. In 1978, Deem and Zabusky \cite{DZ} firstly discovered that there exist simply connected rotating vortex patches with a $m$-fold symmetry for $m\geq 2$ by numerical methods. Later in \cite{Z} Zabusky conjectured that: for any steady (or rotating) system of point vortices there exists a family of steady (or rotating) vortex patches shrinking to these point vortices as the vorticity strength goes to infinity.
Burbea in \cite{Burb} partially answered Zabusky's question by using bifurcation theory. In 1988, Wan \cite{W} studied Zabusky's conjecture and proved that for any rotating system of point vortices with some non-degenerate conditions, the conjecture is ture. Moreover, he also analyzed the linear stability of these rotating vortex patches. Recently in \cite{DE} the authors proved existence of doubly connected rotating vortex patches with a $m$-fold symmetry for some $m\geq 3.$ For active scaler equations, existence of corotating and counter-rotating vortex patches is proved in \cite{HM}.

All of the results mentioned above are about the whole plane. As to the disk, we recall the result in \cite{dHHM}, where the authors studied existence of rotating vortex patches with $m$-fold symmetry bifurcating from a circular patch or an annulus patch. We also point out that in \cite{Hm} the author studied the radial symmetry property of rotating patches in the disk.
Our aim in this paper is to construct another type of rotating vortex patches in the disk. For simplicity we only consider the unit disk centered at the origin, denoted by $D=\{x\in\mathbb R^2\mid |x|<1\}$. The Euler equations in $D$ with impermeability boundary condition is
\begin{equation}\label{1}
\begin{cases}
  \partial_t\mathbf{v}(x,t)+(\mathbf{v}\cdot\nabla)\mathbf{v}(x,t)=-\nabla P(x,t) &\text{in $D\times(0,+\infty)$},\\
  \nabla\cdot\mathbf{v}(x,t)=0 &\text{in $D\times(0,+\infty)$},\\
 \mathbf{v}(x,0)=\mathbf{v}_0(x)&\text{in $D$},
 \\ \mathbf{v}(x,t)\cdot \vec{n}(x)=0 &\text{on $\partial D\times(0,+\infty)$},
\end{cases}
\end{equation}
where $\vec{n}(x)$ is the outward unit normal at $x\in\partial D$.
In this situation, we still have the following vorticity equation
\begin{equation}\label{3}
\begin{cases}
\partial_t\omega(x,t)+\nabla\cdot(\mathbf{v}\omega)(x,t)=0 &\text{in $D\times(0,+\infty)$},\\
\omega(\cdot,0)=\omega_0:=curl\mathbf{v}_0 &\text{in } D.
\end{cases}
\end{equation}
Since $\mathbf{v}$ is divergence-free and $\mathbf{v}\cdot \vec{n}=0$ on $\partial D$, $\mathbf{v}$ can be expressed in terms of $\omega$
\begin{equation}\label{23}
\mathbf{v}=\nabla^\perp G\omega=(\partial_2G\omega,-\partial_1G\omega),
\end{equation}
 where $G\omega(x,t)=\int_DG(x,y)\omega(y,t)dy$,  and $G$ is the Green's function for $-\Delta$ in $D$ with zero
Dirichlet boundary condition in $D$, that is,
\begin{equation}
G(x,y)=-\frac{1}{2\pi}\ln |x-y|-h(x,y), \,\,\,x,y\in
D,
\end{equation}
where $h(x,y)=-\frac{1}{2\pi}\ln|y|-\frac{1}{2\pi}\ln\big|{x}-\frac{{y}}{|{y}|^2}\big|$ is the regular part of $G(x,y)$.

From \eqref{3} and \eqref{23}, by integration by parts we give the definition of weak solutions to the vorticity equation \eqref{3}.
\begin{definition}
Suppose $p\in[4/3,+\infty]$. We call $\omega(x,t)\in L^\infty((0,+\infty);L^p(D))$ a weak solution to \eqref{3} if
 \begin{equation}\label{997}
  \int_D\omega_0(x)\xi(x,0)dx+\int_0^{+\infty}\int_D\omega(\partial_t\xi+\nabla\xi\cdot \nabla^\perp G\omega)dxdt=0
  \end{equation}
for all $\xi\in C_c^{\infty}(D\times[0,+\infty))$.
\end{definition}
Note that for $\omega\in L^\infty((0,+\infty); L^{4/3}(D))$, we have $G\omega\in L^\infty((0,+\infty); W^{2,{4/3}}(D))$ by $L^p$ estimate, thus $\nabla G\omega\in L^\infty((0,+\infty); L^4(D))$ by Sobolev embedding. So the integral in \eqref{997} makes sense by H\"older's inequality.

The existence and uniqueness result for the vorticity equation when $p=+\infty$ is firstly proved by Yudovich \cite{Y}. For general $p>4/3$, by using an approximation procedure and the DiPerna-Lions theory of linear transport equations \cite{DL}, Burton \cite{B5} proved the following theorem.

\begin{theoremA}\label{A}
Suppose $4/3<p<+\infty$ and $\omega_0\in L^p(D)$. Then there exists a weak solution $\omega(x,t)\in L^\infty((0,+\infty);L^p(D))$ to the vorticity equation \eqref{3}. Moreover,
\begin{itemize}
\item[(i)] all $L^\infty( (0,+\infty); L^p(D))$ solutions belong to $C( [0,+\infty); L^p(D))$;
\item[(ii)] for any weak solution $\omega(x,t)\in L^{\infty}((0,+\infty);L^p(D))$,  we have $\omega(x,t)\in R_{\omega_0}$ for all $t\geq 0$, where $R_{\omega_0}$ denotes the rearrangement class of $\omega_0$,
 \begin{equation}
 R_{\omega_0}:=\{ v \in L^1_{loc}(D)\mid  |\{v>a\}|=|\{\omega_0>a\}| ,\forall a\in \mathbb R\};
 \end{equation}
\item[(iii)] for any $L^\infty( (0,+\infty); L^p(D))$ solutions, the angular momentum is conserved, or equivalently,
\[J(t)=J(0),\,\,\forall t\in[0,+\infty),\,\,\text{where }J(t):=\int_D|x|^2\omega(x,t)dx;\]
\item[(iv)] if $p\geq 3/2$, then the kinetic energy of the fluid is conserved, or equivalently,
\[E(t)=E(0),\,\,\forall t\in[0,+\infty),\,\,\text{where }E(t):=\frac{1}{2}\int_D\int_DG(x,y)\omega(x,t)\omega(y,t)dxdy;\]
\end{itemize}
\end{theoremA}

Inspired by the study of rotating vortex patches in the whole plane, our aim in this paper is to construct a vortex patch solution $\omega(x,t)$ in $D$ satisfying
\begin{equation}
\omega(x,t)=w(e^{-i\Omega t}x),\,\,w=\lambda I_{A_0},
\end{equation}
where $\Omega$ represents the angular velocity.
By a simple calculation, it is easy to check $w$ satisfies
\begin{equation}\label{w}
\nabla\cdot\left(w\nabla^\perp(Gw+\frac{\Omega}{2}|x|^2)\right)=0.
\end{equation}
The weak form of \eqref{w} is
\begin{equation}\label{001}
\int_Dw(x)\nabla\left(Gw(x)+\frac{\Omega}{2}|x|^2\right)\cdot \nabla^\perp \phi(x) dx=0
  \end{equation}
for all $\phi\in C_c^{\infty}(D)$.

To find a vortex patch solution satisfying \eqref{001}, we use the vorticity method established by Arnold \cite{A}, which asserts that a steady flow can be seen as a constrained critical point of the kinetic energy, and the flow is stable if and only if this critical point is non-degenerate. A good reference in this respect is \cite{AK}. The vorticity method was later developed by many authors. See \cite{B2,B3,B4,Ta,T,WP}.  The method used in this paper is closely related to \cite{T}. In \cite{T}, Turkington solved a variational problem for the vorticity to obtain existence of steady vortex patches in general bounded domains. Let $D_0$ be a simply connected domain with a smooth boundary, $G_0$ be the Green's function for $-\Delta$ in $D_0$ with zero boundary condition. Consider the maximization the kinetic energy
\[{E}(\omega):=\frac{1}{2}\int_{D_0}\int_{D_0}G_0(x,y)\omega(x)\omega(x)dxdy\]
in the admissible class
\begin{equation}\label{K}
K_\lambda(D_0):=\{\omega\in L^\infty(D_0)\mid 0\leq\omega\leq\lambda\text{ a.e. in }D_0,\,\,\int_{D_0}\omega(x)dx=1\}.
\end{equation}
Turkington proved that there exists a maximizer for $E$ over $K_\lambda(D_0)$, and any maximizer $\omega^\lambda$ must be a steady vortex patch with the form $\omega^\lambda=\lambda I_{U^\lambda}$. Moreover, he showed that as $\lambda\rightarrow+\infty$, the vortex core $U^\lambda$ shrinks to a global minimum point of the Robin function of $D_0$, and the scaled version of $\partial U^\lambda$ converges to the unit circle in $C^1$ sense.

Inspired by Turkington's method, we consider the maximization of the following functional
\begin{equation}\label{E}
\mathcal{E}(w):=\frac{1}{2}\int_{D}\int_{D}G(x,y)w(x)w(x)dxdy+\frac{\Omega}{2}\int_D|x|^2w(x)dx
\end{equation}
in the admissible class $K_\lambda(D)$
\begin{equation}\label{KD}
K_\lambda(D):=\{w\in L^\infty(D)\mid 0\leq w\leq\lambda\text{ a.e. in }D,\,\,\int_{D}w(x)dx=1\}.
\end{equation}
It is easy to prove that there exists a maximizer of $\mathcal{E}$ over $K_\lambda(D)$ but with the form $w^\lambda=\lambda I_{A^\lambda}+2\Omega I_{B^\lambda}$, where
 \[A^\lambda=\{x\in D\mid Gw^\lambda(x)+\frac{\Omega}{2}|x|^2>\mu^\lambda\}\text{ and }B^\lambda=\{x\in D\mid Gw^\lambda(x)+\frac{\Omega}{2}|x|^2=\mu^\lambda\}\]
  for some $\mu^\lambda\in\mathbb R$ depending on $\lambda$. If $\lambda=2\Omega$, then obviously $w^\lambda$ is still a vortex patch solution. If $\lambda\neq2\Omega$, we expect $|B^\lambda|=0,$ but it is hard to prove this by using Turkington's technique. To circumvent this difficulty, we use the strict convexity of the functional $\mathcal{E}$ to conclude that the any maximizer $w^\lambda$ is in fact the unique maximizer of the functional
  \begin{equation}\label{qqq}
\mathcal{Q}(w):=\int_D(Gw^\lambda(x)+\frac{\Omega}{2}|x|^2)w(x)dx
  \end{equation}
in the admissible $K_\lambda(D)$. From this fact, we can easily deduce that the measure of $B^\lambda$ is zero if $\lambda\neq 2\Omega$. See Proposition \ref{patch} in Section 2. The fact that any maximizer of $\mathcal{E}$ over $K_\lambda(D)$ is a vortex patch solution will be used to prove Theorem \ref{os} below.
In addition, we also analyze the limiting behavior of $w^\lambda$ as $\lambda\rightarrow+\infty$.

The first result of this paper is as follows.
\begin{theorem}\label{Thm}
Let $\Omega,\lambda$ be two positive numbers with $\lambda>|D|^{-1}$, and $\mathcal{E},{K_\lambda(D)}$ be defined by \eqref{E} and \eqref{KD}. Then $\mathcal{E}$ attains its maximum in $K_\lambda(D)$ and any maximizer satisfies \eqref{001}. Moreover,
 any maximizer $w^\lambda$ has the following form
\begin{equation}
w^\lambda=\lambda I_{A^\lambda}+2\Omega I_{B^\lambda},
\end{equation}
where
\[A^\lambda=\{x\in D\mid Gw^\lambda(x)+\frac{\Omega}{2}|x|^2>\mu^\lambda\}\text{ and }B^\lambda=\{x\in D\mid Gw^\lambda(x)+\frac{\Omega}{2}|x|^2=\mu^\lambda\},\]
and $\mu^\lambda$ is the Lagrange multiplier depending on $\lambda$. If $\lambda\neq2\Omega$,  then $|B^\lambda|=0$. Furthermore, as $\lambda\rightarrow+\infty$, the following estimates hold true:
\begin{itemize}
\item[(i)] $diam(A^\lambda)\le R_0\varepsilon$, where $R_0>1$ does not depend on $\lambda$ and $\varepsilon=(\pi\lambda)^{-1/2}$;
\item[(ii)] up to a subsequence, $\int_Dxw^\lambda(x)dx\rightarrow X^*\in D$, where $X^*$ is a global minimum point of $H(x)-\frac{\Omega}{2}|x|^2$, where $H(x):=\frac{1}{2}h(x,x)$ is the Robin function of $D$;
 \item[(iii)]${\lambda}^{-1}w^\lambda(X^\lambda+\varepsilon y)\to I_{B_1(0)}$ weakly star in $L^\infty(B_{R_0}(0))$;
 \item[(iv)] $\pi \psi^\lambda(X^\lambda+\varepsilon y)\to V^* ~\text{in}~ C^1_{loc}(\mathbb{R}^2)$, where $\psi^\lambda:=Gw^\lambda(x)+\frac{\Omega}{2}|x|^2-\mu^\lambda$ and $V^*$ is the Rankine streamfunction defined by
     \begin{equation}\label{rank}
\begin{split}
\ V^*(y)
        :=\left \{
          \begin{array}{cccccc}
              \frac{1}{4}(1-|y|^2),                                                  &  0 \le|y|\le 1,\\
              \frac{1}{2}\ln({|y|}^{-1}),                                            &  1<|y|<\infty.
          \end{array}
          \right.
\end{split}
\end{equation}
\end{itemize}

\end{theorem}

\begin{remark}\label{minim}
For $D$, the unit disk centered at the origin, the Robin function has an explicit expression
\[H(x)=-\frac{1}{4\pi}\ln(1-|x|^2).\]
It is easy to check that when $0<\Omega\leq1/\pi$, the unique minimum point of $H-\frac{\Omega}{2}|x|^2$ in $D$ is the origin.
When $\Omega>\frac{1}{2\pi}$, $H-\frac{\Omega}{2}|x|^2$ attains its minimum in $D\setminus\{{0}\}$ and all the minimum points are on the circle $\{x\in D\mid |x|=(1-(2\pi\Omega)^{-1})^{1/2}\}$.
\end{remark}
\begin{remark}
By (i) and (ii) in Theorem \ref{Thm}, as $\lambda\rightarrow+\infty$, the limit of $w^\lambda$ is a Dirac measure with unit strength at $X^*$~in the distributional sense. By Remark \ref{minim}, $X^*\neq {0}$ if and only if $\Omega>\frac{1}{2\pi}.$ This is consistent with the point vortex model. In fact, according to the point vortex model(see \cite{L}), the motion of a point vortex is described by the following Kirchhoff-Routh equation
\begin{equation}
 \frac{dx(t)}{dt}=-\nabla^\perp H(x(t)).
\end{equation}
It is easy to check that the angular velocity of the point vortex at $X^*$ is $(2\pi(1-|X^*|^2))^{-1}\in (\frac{1}{2\pi},+\infty).$
\end{remark}

\begin{remark}
It is easy to see that the function $u(x):=Gw^\lambda(x)+\frac{\Omega}{2}|x|^2-\frac{\Omega}{2}$ satisfies the following semilinear elliptic equation
\begin{equation}\label{str}
\begin{cases}
-\Delta u=f(u),&\text{in } D,\\
u=0,&\text{on }\partial D,
\end{cases}
\end{equation}
where $f(u)=\lambda I_{\{x\in D\mid u(x)>\mu^\lambda-{\Omega}/{2}\}}-2\Omega.$
In fact, one can construct steady Euler flows by solving \eqref{str} directly. See \cite{AS,CLW,CPY,SV} for example. It is worth mentioning that in \cite{SV} Smets and Schaftingen proved existence of a rotating Euler flow in a disk. However, the flow they constructed is smooth.

\end{remark}

Since we have constructed a solution $w^\lambda$ satisfying \eqref{001}, it is easy to verify that $\omega^\lambda(x,t):=w^\lambda(e^{-i\Omega t}x)$
is a weak solution to the vorticity equation \eqref{3}, rotating in $D$ with angular velocity $\Omega$. Moreover, for any fixed time $t>0$, the support of $\omega^\lambda(x,t)$ ``shrinks" to a  point $X(t)$ as $\lambda\rightarrow+\infty$ in the following sense:
\begin{equation*}
  \begin{split}
  &diam(supp(\omega^{\lambda}(\cdot,t)))\le R_0\varepsilon,\\
 \int_Dxw^\lambda&(x)dx\rightarrow X(t)\text{ (up to a subsequence)},
  \end{split}
\end{equation*}
where $X(t)$ is the solution to the following Kirchhoff-Routh equation
\[\frac{dX(t)}{dt}=-\nabla^\perp H(X(t)),\,\,X(0)=X^*.\]

The second result of this paper is concerned with the orbital stability of the set of maximizers of $\mathcal{E}$ in $K_\lambda$. Define
\begin{equation}\label{s}
\mathcal{S}_\lambda:=\{\omega\in K_\lambda(D)\mid \mathcal{E}(\omega)=\sup_{K_\lambda(D)}\mathcal{E}\}.
\end{equation}
According to Theorem \ref{Thm}, $\mathcal{S}_\lambda$ is not empty, moreover, any element in $\mathcal{S}_\lambda$ is a vortex patch.
By energy and angular momentum conservation in Theorem A, it is also easy to see that for any $\omega_0\in\mathcal{S}_\lambda$, we have $\omega_t\in\mathcal{S}_\lambda$ for all $t>0$, where $\omega_t$ is a weak solution to the vorticity equation with initial vorticity $\omega_0$. An interesting question is, for any given initial vorticity $\omega_0$ that is sufficiently close to $\mathcal{S}_\lambda$ in some norm, will it be close to $\mathcal{S}_\lambda$ for all $t>0$ in the same norm? If it is true, $\mathcal{S}_\lambda$ is said to be orbitally stable.

There are many results concerning the stability of planar vortex flows in the past few decades. See \cite{B5,B6,CW0,CW,Ta,W,WP} and the references listed therein.
The type of stability we consider here is nonlinear stability, which is usually a very difficult problem in hydrodynamics. A very effective method to prove nonlinear stability for smooth planar Euler flows is established by Arnold \cite{A2}, which was later extended to non-smooth flows, for example, vortex patches. See \cite{CW0,CW,Ta,WP}. In \cite{B5}, Burton proved a very general stability criterion for vortex flows in bounded domains, asserting that any steady vortex flow as the strict local maximizer of the kinetic energy on some given rearrangement class is stable in $L^p$ norm. Based on the similar idea, nonlinear orbital stability for vortex pairs in the whole plane was proved in \cite{B6}. The method used in this paper is mostly inspired by \cite{B5,B6}.

The orbital stability of $\mathcal{S}_\lambda$ is stated as follows.

\begin{theorem}\label{os}
Let $\frac{3}{2}\leq p<+\infty$, $\lambda>|D|^{-1}$, and $S_\lambda$ be defined by \eqref{s}. Then $\mathcal{S}_\lambda$ is orbitally stable in $L^p$ norm, or equivalently, for any $\varepsilon>0$, there exists a $\delta>0$, such that for any $\omega_0\in L^p(D)$ satisfying $dist_p(\omega_0,\mathcal{S}_\lambda)<\delta$,  we have  $dist_p(\omega_t,\mathcal{S}_\lambda)<\varepsilon$ for all $t>0$, where $\omega_t$ is a weak solution to the vorticity equation with initial vorticity $\omega_0$, and $dist_p(\omega_0,\mathcal{S}_\lambda)$ is defined by
\begin{equation}
dist_p(\omega_0,\mathcal{S}_\lambda):=\inf_{\omega\in\mathcal{S}_\lambda}\|\omega_0-\omega\|_{L^p(D)}.
\end{equation}
\end{theorem}
To prove Theorem \ref{os}, the key point is compactness. In \cite{B6}, compactness was obtained by a Concentration-Compactness argument. In this paper, compactness comes from the fact any maximizer must be a vortex patch(see Lemma \ref{com} in Section 3). The same idea was also used in \cite{CWW} to prove nonlinear orbital stability for steady vortex patches.

\section{Proof of Theorem \ref{Thm}}
In this section we give the proof of Theorem \ref{Thm}. As mentioned in Section 1, we consider the maximization of $\mathcal{E}$ in $K_\lambda(D)$, where $\mathcal{E}$ and $K_\lambda(D)$ are defined by \eqref{E} and \eqref{K}.
Note that by Fubini's theorem and integration by parts, we have for any $w\in K_\lambda(D)$
\begin{equation}\label{301}
\begin{split}
 \mathcal{E}(w)&=\frac{1}{2}\int_D\int_DG(x,y)w(x)w(y)dxdy+\frac{\Omega}{2}\int_D|x|^2w(x)dx\\
 &=\frac{1}{2}\int_DGw(x)w(x)dx+\frac{\Omega}{2}\int_D|x|^2w(x)dx\\
 &=\frac{1}{2}\int_D|\nabla Gw(x)|^2dx+\frac{\Omega}{2}\int_D|x|^2w(x)dx.
\end{split}
\end{equation}
We also assume throughout this paper that $\lambda > 1/|D|$ such that $K_\lambda(D)$ is not empty.

An absolute maximizer for $\mathcal{E}$ over $K_\lambda(D)$ can be easily found by the direct method. Indeed, we have

\begin{proposition}\label{802}
There exists $w^\lambda \in K_\lambda(D) $ such that
\begin{equation}\label{303}
 \mathcal{E}(w^\lambda)= \sup_{{w} \in K_\lambda(D)}\mathcal{E}({w}).
\end{equation}
\end{proposition}
\begin{proof}
Firstly we show that $K_\lambda(D)$ is sequentially compact in $L^\infty(D)$ in the weak star topology.
In fact, since $K_\lambda(D)$ is a closed and convex subset of $L^2(D)$ in the strong topology, we conclude from Mazur's lemma that $K_\lambda(D)$ is closed in the weak topology of $L^2(D)$, which implies that $K_\lambda(D)$ is closed in $L^\infty(D)$ in the weak star topology.

Now we prove that $\mathcal{E}$ is a sequentially weakly star continuous functional in $L^\infty(D)$. Let $\{w_n\}$ be a sequence in $L^\infty(D)$ such that $w_n\rightarrow w$ weakly star in $ L^\infty(D)$ as $n\rightarrow+\infty$. Then it is easy to see that $w_n\rightarrow w$ weakly in $ L^p(D)$
for any $1<p<+\infty$. By $L^p$ estimate we have $Gw_n\rightarrow Gw$ in $ C^1(\overline{D})$. Taking into account \eqref{301} we get $\lim_{n\to\infty}E(w_n)=E(w)$.

Since $G(x,y)\in L^1(D\times D)$, it follows that $E$ is bounded from above in $K_\lambda(D)$, that is, $\sup_{w\in K_{\lambda}(D)}\mathcal{E}(w)<+\infty$. Then we can take a sequence $\{w_n\}$ such that $\lim_{n\rightarrow}\mathcal{E}(w_n)=\sup_{w\in K_{\lambda}(D)}\mathcal{E}(w)$.  Without loss of generality, we assume that $w_n\rightarrow w^\lambda$ weakly star in $ L^\infty(D)$ for some $w^\lambda\in K_{\lambda}(D)$ as $n\rightarrow+\infty$. It follows easily from the above discussion that $\mathcal{E}(w^\lambda)= \sup_{{w} \in K_\lambda(D)}\mathcal{E}({w})$.
\end{proof}

In the following lemma, by choosing suitable test functions we study the profile of $w^\lambda$.
\begin{lemma}\label{805}
For any maximizer $w^\lambda$ obtained in Lemma $\ref{802}$, we have
\begin{equation}\label{304}
 w^\lambda=\lambda I_{A^\lambda} + 2\Omega I_{B^\lambda}  ~\text{~ a.e. in} ~D,
\end{equation}
where
\begin{equation}\label{AB}
 A^\lambda=\{x\in D\mid Gw^\lambda(x)+\frac{\Omega}{2}|x|^2>\mu^\lambda\} ~\text{and}~ B^\lambda=\{x\in D\mid Gw^\lambda(x)+\frac{\Omega}{2}|x|^2=\mu^\lambda\},
 \end{equation}
and the Lagrange multiplier $\mu^\lambda>0$ is determined by $w^\lambda$ as follows
\begin{equation}\label{305}
\begin{split}
  \mu^\lambda &={\sup}_{\{x\in D\mid w(x)<\lambda\}}\left(Gw(x)+\frac{\Omega}{2}|x|^2\right)
       ={\inf}_{\{x\in D\mid w(x)>0\}}\left(Gw(x)+\frac{\Omega}{2}|x|^2\right).
\end{split}
\end{equation}
\end{lemma}

\begin{proof}
Define a family of test functions $w_s=w^\lambda+s(z_0-z_1)$, $s>0$, where $z_0$ and $z_1$ satisfy
\begin{equation}
\begin{cases}
z_0,z_1\in L^\infty(D),\,\,z_0,z_1\geq 0\text{ a.e. in }D,

 \\ \int_Dz_0(x)dx=\int_D z_1(x)dx,

 \\z_0=0 \quad\text{in } D\setminus\{x\in D\mid w^\lambda(x)\leq\lambda-\delta\},
 \\z_1=0 \quad\text{in } D\setminus\{x\in D\mid w^\lambda(x)\geq\delta\}.
\end{cases}
\end{equation}
Here $\delta$ is a small positive number. It is easy to see that for fixed $z_0,z_1$ and $\delta$, if $s$ is sufficiently small, $w_s$ belongs to $   K_\lambda(D)$. Since $w^\lambda$ is a maximizer, we have
\[0\geq\frac{d\mathcal{E}(w_s)}{ds}\bigg|_{s=0^+}=\int_Dz_0(x)\left(Gw^\lambda(x)+\frac{\Omega}{2}|x|^2\right)dx-\int_Dz_1(x)\left(Gw^\lambda(x)
+\frac{\Omega}{2}|x|^2\right)dx.\]
By the choice of $z_0$ and $z_1$ we deduce that
\begin{equation}\label{1-103}
\sup_{\{x\in D\mid w^\lambda(x)<\lambda\}}\left(Gw^\lambda+\frac{\Omega}{2}|x|^2\right)\leq\inf_{\{x\in D\mid w^\lambda(x)>0\}}\left(Gw^\lambda+\frac{\Omega}{2}|x|^2\right).
\end{equation}
By the continuity of $Gw^\lambda+\frac{\Omega}{2}|x|^2$ , \eqref{1-103} is in fact an equality, that is,
\begin{equation}
\sup_{\{x\in D\mid w^\lambda(x)<\lambda\}}\left(Gw^\lambda+\frac{\Omega}{2}|x|^2\right)=\inf_{\{x\in D\mid w^\lambda(x)>0\}}\left(Gw^\lambda+\frac{\Omega}{2}|x|^2\right).
\end{equation}
Set
\[\mu^\lambda:=\sup_{\{x\in D\mid w^\lambda(x)<\lambda\}}\left(Gw^\lambda+\frac{\Omega}{2}|x|^2\right)=\inf_{\{x\in D\mid w^\lambda(x)>0\}}\left(Gw^\lambda+\frac{\Omega}{2}|x|^2\right).\]
It is easy to check that
\begin{equation}
\begin{cases}
w^\lambda=0\text{\,\,\,\,\,\,a.e.\,} \text{in }\{x\in D\mid Gw^\lambda(x)+\frac{\Omega}{2}|x|^2<\mu^\lambda\},
 \\ w^\lambda=\lambda\text{\,\,\,\,\,\,a.e.\,} \text{in }\{x\in D\mid Gw^\lambda(x)+\frac{\Omega}{2}|x|^2>\mu^\lambda\}.
\end{cases}
\end{equation}
On the level set $\{x\in D\mid Gw^\lambda(x)+\frac{\Omega}{2}|x|^2=\mu^\lambda\}$, by the property of Sobolev functions, we have $-\Delta (G\omega^\lambda+\frac{\Omega}{2}|x|^2)=0\text{\,\,a.e.}$ on $\{x\in D\mid Gw^\lambda(x)+\frac{\Omega}{2}|x|^2=\mu^\lambda\}$, from which we obtain
$w^\lambda=2\Omega$ a.e. on $\{x\in D\mid Gw^\lambda(x)+\frac{\Omega}{2}|x|^2=\mu^\lambda\}$. The proof is completed.

\end{proof}

\begin{remark}
In Lemma \ref{805}, we only show that for fixed $w^\lambda$ the Lagrange multiplier $\mu^\lambda $ is unique, however, the mapping from $\lambda$ to $\mu ^\lambda$ may be multiple-valued.
\end{remark}

\begin{proposition}\label{patch}
Suppose $w^\lambda$ is a maximizer and $\lambda\neq2{\Omega}$, then $|B^\lambda|=0$.
\end{proposition}
\begin{proof}
We divide the proof into three steps.

\textbf{Step 1:} For any $w_1,w_2\in K_{\lambda}(D)$, we have
\begin{equation}\label{conve} \int_D\int_DG(x,y)w_1(x)w_2(y)dxdy\leq E(w_1)+E(w_2),\end{equation}
where\[E(w):=\frac{1}{2}\int_D\int_DG(x,y)w(x)w(y)dxdy,\,\,w\in K_\lambda(D),\]
and the equality holds if and only if $w_1=w_2$. In fact, we need only to observe that $E(w_1-w_2)\geq0$, and $E(w_1-w_2)=0$ if and only if $w_1=w_2.$
Combining the symmetry of the Green's function, we get \eqref{conve}.

\textbf{Step 2:} $w^\lambda$ is the unique maximizer of the following functional
  \begin{equation}\label{qqqq}
\mathcal{Q}(w):=\int_D(Gw^\lambda(x)+\frac{\Omega}{2}|x|^2)w(x)dx
  \end{equation}
in the admissible $K_\lambda(D)$. In fact, by Step 1 we have
\begin{equation}\label{uniq}
\begin{split}
\mathcal{Q}(w)&=\int_DGw^\lambda(x)w(x)dx+\int_D\frac{\Omega}{2}|x|^2w(x)dx\\
&\leq E(w^\lambda)+E(w)+\int_D\frac{\Omega}{2}|x|^2w(x)dx\\
&\leq E(w^\lambda)+\mathcal{E}(w^\lambda)\\
&=\mathcal{Q}(w^\lambda).
\end{split}
\end{equation}
Moreover, the equality holds if and only if $w=w^\lambda$, which is the desired result.

\textbf{Step 3:} If $\lambda\neq2{\Omega}$, then $|B^\lambda|=0$. In fact, if $\lambda<2\Omega,$ the conclusion is obvious. So we need only to prove the case $\lambda>2\Omega.$ Suppose $|B^\lambda|\neq0$. We define $\bar{w}=\lambda I_{A^\lambda}+2^{-1}({\lambda}+2\Omega) I_{C^\lambda}$, where $C^\lambda$ satisfying $C^\lambda\subset B^\lambda$ and $|C^\lambda|=4\Omega(\lambda+2\Omega)^{-1}|B^\lambda|$. Then it is easy to check that $\bar{w}\in K_\lambda(D)$ and $\bar{w}\neq w^\lambda$. But we have
\begin{equation}
\begin{split}
\mathcal{Q}(\bar{w})&=\int_D(Gw^\lambda(x)+\frac{\Omega}{2}|x|^2)w(x)dx\\
&=\lambda\int_{A^\lambda}(Gw^\lambda(x)+\frac{\Omega}{2}|x|^2)dx+\frac{({\lambda}+2\Omega)}{2}\int_{C^\lambda}(Gw^\lambda(x)+\frac{\Omega}{2}|x|^2)dx\\
&=\lambda\int_{A^\lambda}(Gw^\lambda(x)+\frac{\Omega}{2}|x|^2)dx+\frac{({\lambda}+2\Omega)}{2}|{C^\lambda}|\mu^\lambda \\
&=\lambda\int_{A^\lambda}(Gw^\lambda(x)+\frac{\Omega}{2}|x|^2)dx+2\Omega|B^\lambda|\mu^\lambda \\
&=\lambda\int_{A^\lambda}(Gw^\lambda(x)+\frac{\Omega}{2}|x|^2)dx+2\Omega\int_{B^\lambda}(Gw^\lambda(x)+\frac{\Omega}{2}|x|^2)dx\\
&=\mathcal{Q}(w^\lambda),
\end{split}
\end{equation}
which is a contradiction to Step 2.

\end{proof}
From Lemma \ref{805} and Proposition \ref{patch}, we can easily deduce the following
\begin{corollary}\label{coro}
For any $\lambda>|D|^{-1}$, any maximizer $w^\lambda$ has the from $w^\lambda=\lambda_{\tilde{A}^\lambda}$ a.e. for some $\tilde{A}^\lambda\subset D$.
\end{corollary}

 In the following we analyze the limiting behavior of $w^\lambda$ as $\lambda\rightarrow+\infty$.
For simplicity, we will use $C$ to denote various positive numbers independent of $\lambda$.
 \begin{lemma}\label{808}
$\mathcal{E}(w^\lambda)=-({4\pi})^{-1}\ln{\varepsilon}+O(1),$ where $\varepsilon$ satisfies $\lambda \pi \varepsilon^2=1$.
\end{lemma}

\begin{proof}
Choose a test function $\hat{w}\in K_\lambda(D)$ defined by
\begin{equation}\label{311}
\hat{w}=\lambda I_{B_\varepsilon(0)}.
\end{equation}
Since $w^\lambda$ is a maximizer, we have $\mathcal{E}(w^\lambda)\ge \mathcal{E}(\hat{w})$. By a simple calculation, we obtain
\begin{equation*}
\begin{split}
      \mathcal{E}(\hat{w})&=\frac{1}{2}\int_D{\int_DG(x,y)\hat{w}(x)\hat{w}(y)dx}dy+\frac{\Omega}{2}\int_D|x|^2\hat{w}(x)dx\\
      &\geq \frac{1}{2}\int_D\int_D-\frac{1}{2\pi}\ln|x-y|\hat{w}(x)\hat{w}(y)dxdy-C\\
      &=\frac{\lambda^2}{2}\int_{B_\varepsilon(0)}\int_{B_\varepsilon(0)}-\frac{1}{2\pi}\ln|x-y|dxdy-C\\
      &\geq\frac{\lambda^2}{2}\int_{B_\varepsilon(0)}\int_{B_\varepsilon(0)}-\frac{1}{2\pi}\ln(2\varepsilon)dxdy-C\\
      &= -({4\pi})^{-1}\ln{\varepsilon}-C,
\end{split}
\end{equation*}
where we used the fact \[\int_D\int_Dh(x,y)\hat{w}(x)\hat{w}(y)dxdy\rightarrow h(0,0) \text{ as }\lambda\rightarrow +\infty.\]
On the other hand,
\begin{equation*}
\begin{split}
      \mathcal{E}(w^\lambda)&=\frac{1}{2}\int_D{\int_DG(x,y)w^\lambda(x)w^\lambda(y)dx}dy+\frac{\Omega}{2}\int_D|x|^2w^\lambda(x)dx\\
      &\leq \frac{1}{2}\int_D\int_D-\frac{1}{2\pi}\ln|x-y|\hat{w}(x)\hat{w}(y)dxdy+C\\
      &\leq -({4\pi})^{-1}\ln{\varepsilon}+C,
\end{split}
\end{equation*}
where we used Riesz's rearrangement inequality(see \cite{LL}, 3.7) and the fact that $h(x,y)$ is bounded from below in $D\times D$.~The proof is completed.
\end{proof}

Now we estimate the energy of the ``vortex core". Define $\psi^\lambda = Gw^\lambda+\frac{\Omega}{2}|x|^2-\mu^\lambda$. The kinetic energy of the ``vortex core" is defined as follows:
\begin{equation}\label{309}
T(w^\lambda)=\frac{1}{2}\int_D{|\nabla\psi^\lambda_+(x)|^2}dx,
\end{equation}
where $\psi^\lambda_+= \max\{\psi^\lambda, 0\}$.

\begin{lemma}\label{810}
$T(w^\lambda)\le  C.$
\end{lemma}

\begin{proof}
Firstly it is easy to check that $\psi^\lambda$ satisfies the following elliptic equation
\begin{equation}\label{308}
\begin{cases}
-\Delta \psi^\lambda = w^\lambda-2\Omega &\text{in}~~~ D ,\\
 ~~~~~~~~~\psi^\lambda=\frac{\Omega}{2}-\mu^\lambda&\text{on}~~~\partial D.
\end{cases}
\end{equation}
Set~$\gamma^\lambda:=\max({\Omega}/{2}-\mu^\lambda,0)\in[0,\Omega/{2}]$. Let us multiply both sides by $\psi^\lambda_+-\gamma^\lambda\in H^1_0(D)$. By integration by parts we have
\begin{equation}\label{vor}
\begin{split}
  2T(w^\lambda)&= \int_D{|\nabla\psi^\lambda_+(x)|^2}dx\\
  &=\int_D(w^\lambda(x)-2\Omega)\psi^\lambda_+(x)dx-\gamma^\lambda\int_D(w^\lambda(x)-2\Omega)dx\\
                                   &\le\int_Dw^\lambda(\psi^\lambda_+(x)-\gamma^\lambda)dx+C\\
                                   &\le\lambda|\{x\in D\mid\psi^\lambda(x)>\gamma^\lambda\}|^{{1}/{2}}\left(\int_D(\psi^\lambda_+(x)-\gamma^\lambda)^2dx\right)^{{1}/{2}} +C\\
                                   &\le C\lambda|\{x\in D\mid\psi^\lambda(x)>0\}|^{{1}/{2}}\int_D{|\nabla\psi^\lambda_+(x)|}dx+C\\
                                   &\le C\lambda|\{x\in D\mid\psi^\lambda(x)>0\}|\left(\int_D{|\nabla\psi^\lambda_+(x)|^2}dx\right)^{{1}/{2}}+C\\
                                   &\le C(T(w^\lambda))^{{1}/{2}}+C,
\end{split}
\end{equation}
where we used H\"older's inequality and Sobolev inequality. From \eqref{vor} we conclude the desired result.
\end{proof}

We are now ready to estimate the Lagrange multiplier $\mu^\lambda$.
\begin{lemma}\label{811}
$\mu^\lambda=-({2\pi})^{-1}\ln{\varepsilon}+O(1)$, as $\lambda\rightarrow+\infty$.
\end{lemma}

\begin{proof}
Set~$\gamma^\lambda:=\max({\Omega}/{2}-\mu^\lambda,0)$.
By \eqref{vor} we have as $\lambda\rightarrow+\infty$
\begin{equation}\label{890}
\begin{split}
\    2T(w^\lambda)&=\int_Dw^\lambda(x)\psi^\lambda_+(x)dx-2\Omega\int_D\psi^\lambda_+(x)dx+O(1)\\
          &=\int_Dw^\lambda(x)\psi^\lambda(x) dx-2\Omega\int_D\psi^\lambda_+(x)dx+O(1)\\
          &=\int_Dw^\lambda(x)(Gw^\lambda(x)+\frac{\Omega}{2}|x|^2-\mu^\lambda)dx-\frac{\Omega}{2}\int_Dw^\lambda(x)|x|^2dx-2\Omega\int_D\psi^\lambda_+(x)dx+O(1)\\
          &=2\mathcal{E}(w^\lambda)-\mu^\lambda-2\Omega\int_D(\psi^\lambda_+(x)-\gamma^\lambda)dx+O(1)\\
          &=2\mathcal{E}(w^\lambda)-\mu^\lambda+O(1).
\end{split}
\end{equation}
Here we used $\int_D|\psi^+-\gamma^\lambda|dx\le C ({\int_D{|\nabla\psi^+|^2}dx})^{\frac{1}{2}}\leq C$. From Lemma $\ref{808}$ and Lemma $\ref{810}$ we get the desired result.
\end{proof}

In the next lemma we show that the diameter of support of $w^\lambda$ is of order $\varepsilon$.
\begin{lemma}\label{812}
There is a constant $R_0>1$ independent of $\lambda$, such that$
\ \text{diam}(\text{supp}(w^\lambda))\le R_0\varepsilon$,~with $\varepsilon$ satisfying $\lambda\pi\varepsilon^2=1$.
\end{lemma}

\begin{proof}
For any $x\in supp(w^\lambda)$, we have $\psi^\lambda(x)\ge0$. Recalling the definition of $\psi^\lambda$, we deduce that
\begin{equation*}
\begin{split}
  \mu^\lambda &\le Gw^\lambda(x)+\frac{\Omega}{2}|x|^2 =\frac{1}{2\pi}\int_D\ln| x-y|^{-1}w^\lambda(y)dy-\int_Dh(x,y)w^\lambda(y)dy+\frac{\Omega}{2}|x|^2.
\end{split}
\end{equation*}
Taking into account Lemma \ref{811} we obtain
\[\frac{1}{2\pi}\ln{\frac{1}{\varepsilon}}-C\le \frac{1}{2\pi}\int_D\ln|x-y|^{-1}w^\lambda(y)dy,\]
That is to say,
$$-2\pi C\le\int_D\ln\frac{\varepsilon}{|x-y|}w^\lambda(y)dy.$$
Now by choosing $R>1$ to be determined, we calculate as follows
\begin{equation}\label{8888}
\begin{split}
\ -2\pi C&\le\int_D\ln\frac{\varepsilon}{|x-y|}w^\lambda(y)dy\\
    &\le\int_{D\cap{B_{R\varepsilon}(x)}}\ln\frac{\varepsilon}{|x-y|}w^\lambda(y)dy+\int_{D\backslash B_{R\varepsilon}(x)}\ln\frac{\varepsilon}{|x-y|}w^\lambda(y)dy\\
    &\le\int_{D\cap B_{\varepsilon}(x)}\ln\frac{\varepsilon}{|x-y|}w^\lambda(y)dy-\ln R\int_{D\backslash B_{R\varepsilon}(x)}w^\lambda(y)dy\\
    &\le\lambda\int_{B_{\varepsilon}(0)}\ln\frac{\varepsilon}{|y|}dy-\ln R\int_{D\backslash B_{R\varepsilon}(x)}w^\lambda(y)dy\\
    &=\frac{1}{2}-\ln R\int_{D\backslash B_{R\varepsilon}(x)}w^\lambda(y)dy.
\end{split}
\end{equation}
From \eqref{8888} we get
\begin{equation*}
  \int_{D\backslash B_{R\varepsilon}(x)}w^\lambda(y)dy\le C(\ln R)^{-1}.
\end{equation*}
Taking $R>1$ large enough such that $C(\ln R)^{-1}<1/2$, we obtain
\begin{equation}\label{316}
 \int_{D\cap{B_{R\varepsilon}(x)}}w^\lambda(y)dy>\frac{1}{2}.
\end{equation}
Now the lemma is proved by taking $R_0=2R.$ In fact, suppose $diam (supp(w^\lambda))> 2R\varepsilon$, then there exist $x_1, x_2\in supp(w^\lambda)$ such that $B_{R\varepsilon}(x_1)\cap B_{R\varepsilon}(x_2)=\varnothing$. By $\eqref{316}$,
$$1=\int_Dw^\lambda(y)dy\ge \int_{D\cap{B_{R\varepsilon}(x_1)}}w^\lambda(y)dy+\int_{D\cap{B_{R\varepsilon}(x_2)}}w^\lambda(y)dy>1,$$
which leads to a contradiction.
\end{proof}

We proceed to study the limiting behavior of $w^\lambda$ as $\lambda \to +\infty.$
Define the center of $w^\lambda$ to be
\begin{equation}\label{317}
\  X^\lambda = \int_Dxw^\lambda(x)dx.
\end{equation}
Since $\bar{D}$ is a compact set, for the remainder of the discussion we may fix a sequence $\lambda=\lambda_j\to+\infty$ such that
\begin{equation}\label{318}
 X^\lambda\to X^*\in\bar{D} ~\,\,~~\text{as}~\ ~ \lambda=\lambda_j\to+\infty.
\end{equation}

\begin{lemma}\label{813}
Let $X^*$ be defined by $\eqref{318}$, then
\begin{equation}\label{319}
  H(X^*)-\frac{\Omega}{2}|X^*|^2=\min_{x\in D}(H(x)-\frac{\Omega}{2}|x|^2),
\end{equation}
or equivalently,
\begin{equation}\label{320}
 |X^*|=\begin{cases}\sqrt{1-\frac{1}{\pi\Omega}}, &\text{if }\Omega>1/\pi,\\
 0, &\text{if }\Omega\leq1/\pi.
 \end{cases}
\end{equation}
\end{lemma}

\begin{proof}
For any $\hat{x}\in D$,
we define a test function $\hat{w}^\lambda=\lambda I_{B_\varepsilon(\hat{x})}.$ For sufficiently large $\lambda$, we have $\hat{w}^\lambda\in K_{\lambda}(D)$. So we have by Riesz's rearrangement inequality(see \cite{LL}, 3.2) and the fact that  $\mathcal{E}(w^\lambda)\ge \mathcal{E}(\hat{w}^\lambda)$,
\begin{equation*}
\begin{split}
 \frac{1}{2}\int_D&{\int_Dh(x,y)w^\lambda(x)w^\lambda(y)dx}dy-\frac{\Omega}{2}\int_D|x|^2w^\lambda(x)dx\\
            &=\frac{1}{4\pi}\int_D{\int_D~\ln|x-y|^{-1}w^\lambda (x)w^\lambda(y)dx}dy-\mathcal{E}(w^\lambda)\\
            &\le\frac{1}{4\pi}\int_D{\int_D~\ln|x-y|^{-1}\hat{w}^\lambda (x)\hat{w}^\lambda(y)dx}dy-\mathcal{E}(\hat{w}^\lambda)\\
            &\le\frac{1}{2}\int_D{\int_D~h(x,y)\hat{w}^\lambda(x)\hat{w}^\lambda (y)dx}dy-\frac{\Omega}{2}\int_D|x|^2\hat{w}^\lambda(x)dx.
\end{split}
\end{equation*}
Letting $\lambda\to+\infty$, we deduce that
$$H(X^*)-\frac{\Omega}{2}|X^*|^2\le H(\hat{x})-\frac{\Omega}{2}|\hat{x}|^2,\,\,\forall\,\hat{x}\in D,$$
from which we obtain \eqref{319}. Firstly by a simple calculation, it is easy to check that $X^*$ satisfies \eqref{320}.
\end{proof}

We now turn to study the small scale asymptotics of the $w^\lambda$. To begin with, we state a result from potential theory which will be frequently used later.

\begin{lemma}\label{807}[\cite{T}, Lemma 4.2]
Let $R\in(1,+\infty)$ be a constant. Define the class $\mathcal{K}_R$  as follows
$$\mathcal{K}_R=\{\zeta \in L^\infty(B_R(0)):0\le \zeta \le 1, \int_{B_R(0)}\zeta(x)dx=\pi \}$$
Let the functional $F$ be defined by
\begin{equation}\label{306}
 F(\zeta)=\frac{1}{4\pi}\int_{B_R(0)}{\int_{B_R(0)}N(x,y)w(x)w(y)dx}dy,
\end{equation}
where $N(x,y)=({2\pi})^{-1}\ln|x-y|^{-1}.$
Then $\zeta^*=I_{B_1(0)}$ is the unique maximizer of $F$ over $\mathcal{K}_R$ satisfying
\begin{equation}\label{307}
\int_{B_R(0)}x\zeta(x)dx=0.
\end{equation}
\end{lemma}

\begin{lemma}\label{815}
Let $R_0$ be the positive number obtained in Lemma $\ref{812}$. Then as $\lambda\to +\infty$, we have $\zeta^\lambda\to \zeta^*:=I_{B_1(0)}$ weakly star in $L^\infty(B_{R_0}(0))$, where
\begin{equation}\label{321}
\ \zeta^\lambda(y)=\frac{1}{\lambda}w^\lambda(X^\lambda+\varepsilon y)\in L^\infty(B_{R_0}(0)).
\end{equation}
\end{lemma}

\begin{proof}
Firstly, it is easy to see that $\zeta^{\lambda}\in \mathcal{K}_{R_0}$. Moreover, by the definition of $X^{\lambda}$, we have
\begin{equation}\label{564}
\begin{split}
\int_{B_{R_0}(0)}y \zeta^{\lambda}(y)dy=& \lambda^{-1}\int_{B_{R_0}(0)}y \omega^{\lambda}(X^{\lambda}+\varepsilon y)dy \\
=&\lambda^{-1}\int_{B_{R_0\varepsilon}(X^{\lambda})}\varepsilon^{-1}(x-X^{\lambda})\omega^{\lambda}(x)\varepsilon^{-2}dx \\
=&\pi\varepsilon^{-1}\int_{D}(x-X^{\lambda})\omega^{\lambda}(x)dx \\
=&0,
\end{split}
\end{equation}
that is, the center of $\zeta^{\lambda}$ is 0 for sufficiently large $\lambda$.

Now for any $\tilde{\zeta}\in \mathcal{K}_{R_0}$, define $\tilde{\omega}\in K_{\lambda}(D)$ as follows
\begin{equation}
\tilde{\omega}(x)=\begin{cases}
&\lambda\tilde{\zeta}(\varepsilon^{-1}(x-X^{\lambda})),\,\,\,\,x\in B_{R_0\varepsilon}(X^{\lambda}), \\
&0, \quad\quad \quad\quad\quad \quad\quad \,\,\, x\in D\setminus B_{R_0\varepsilon}(X^{\lambda}).
\end{cases}
\end{equation}
Direct calculation shows that as $\lambda\rightarrow +\infty$,
\begin{equation}
\begin{split}
\mathcal{E}(\tilde{\omega})=&\frac{1}{2}\int_{D}\int_{D}G(x,x')\tilde{\omega}(x)\tilde{\omega}(x')dxdx'+\frac{\Omega}{2}\int_{D}|x|^2\tilde{\omega}(x)dx \\
=&\frac{1}{4\pi}\int_{D}\int_{D}\ln\frac{1}{|x-x'|}\tilde{\omega}(x)\tilde{\omega}(x')dxdx'-\frac{1}{2}\int_{D}\int_{D}h(x,x')\tilde{\omega}(x)\tilde{\omega}(x')dxdx'+\frac{\Omega}{2}\int_{D}|x|^2\tilde{\omega}(x)dx \\
=&\frac{1}{4\pi}\ln\frac{1}{\varepsilon}+\frac{1}{\pi^2}F(\tilde{\zeta})-H(X^*)+\frac{\Omega}{2}|X^*|^2+o(1).
\end{split}
\end{equation}
By a similar calculation for $\omega^{\lambda}$ and $\zeta^{\lambda}$, we also have
\[\mathcal{E}(\omega^{\lambda})=\frac{1}{4\pi}\ln\frac{1}{\varepsilon}+\frac{1}{\pi^2}F(\zeta^{\lambda})-H(X^*)+\frac{\Omega}{2}|X^*|^2+o(1),\]
as $\lambda\rightarrow +\infty$. Since $\mathcal{E}(\tilde{\omega})\leq \mathcal{E}(\omega^{\lambda})$, we obtain as $\lambda\rightarrow +\infty$
\[F(\tilde{\zeta})\leq F(\zeta^{\lambda})+o(1).\]
On the other hand, since $||\zeta^{\lambda}||_{L^{\infty}(B_{R_0}(0))}\leq 1$, there exists $\zeta\in \mathcal{K}_{R_0}$ such that up to subsequence
\[\zeta^{\lambda}\rightarrow\zeta\ \ \ \ \text{weakly star in }\ L^{\infty}(B_{R_0}(0))\]
as $\lambda\rightarrow +\infty$.
By the continuity of $F$, we deduce that $F(\zeta)=\lim_{m\rightarrow+\infty}F(\zeta^{\lambda})\geq F(\tilde{\zeta})$. Since $\tilde{\zeta}\in \mathcal{K}_{R_0}$ is arbitrary and taking into account the fact that
\[\int_{B_{R_0}(0)}y\zeta(y)dy=\lim_{m\rightarrow+\infty}\int_{B_{R_0}(0)}y\zeta^{\lambda}(y)dy=0,\]
we deduce from Lemma \ref{807} that $\zeta=\zeta^*=I_{B_1(0)}$. Finally, since the maximizer of $F(\zeta)$ over $\mathcal{K}_{R_0}$ is unique, the convergence is independent of the choice of any subsequence, which completes the proof.
\end{proof}

To study the limiting behavior of $\psi^\lambda$, we define
\begin{equation}
\ V^\lambda(y):=\pi \psi^\lambda(X^\lambda+\varepsilon y),\,\,y\in D_\lambda:=\{y\in\mathbb{R}^2: X^\lambda+\epsilon y\in D\}.
\end{equation}
By $L^p$ estimate, for any fixed $R_1>R_0$, if $\lambda$ is sufficiently large, we have $V^\lambda \in  C^{1,\alpha}(\overline{{B}_{R_1}(0)})$ for each $0<\alpha<1$.

\begin{lemma}\label{816}
As $\lambda\to +\infty$, we have $V^\lambda\to V^* ~\text{in}~ C^1(\overline{{B}_{R_1}(0)})$, where $V^*$ is defined by \eqref{rank}.
\end{lemma}

\begin{proof}
Firstly, define $\tilde{V}^{\lambda}(y)\in C^1(\mathbb{R}^2)$ by setting
\[\tilde{V}^{\lambda}(y)=\frac{1}{2\pi}\int_{\mathbb{R}^2}\ln|y-y'|^{-1}\zeta^{\lambda}(y')dy'.\]
Since $\text{supp}(\zeta^{\lambda})\subset \overline{B_{R_0}(0)}$ and $0\leq \zeta^{\lambda}\leq 1$ in $B_{R_0}(0)$, by standard elliptic theory, we have,
\begin{equation}
\begin{cases}
&-\Delta\tilde{V}^{\lambda}(y)=\zeta^{\lambda}(y),\ \ \ \ \ \ y\in B_{R_1}(0), \\
&|\nabla\tilde{V}^{\lambda}(y)|\leq C,\ \ \ \ \ \ \ \ \ \ \ \ y\in B_{R_1}(0), \\
&|\nabla\tilde{V}^{\lambda}(y)-\nabla\tilde{V}^{\lambda}(y')|\leq C|y-y'|\ln(1+\frac{2R_1}{|y-y'|}),\ \ y,y'\in B_{R_1}(0).
\end{cases}
\end{equation}
 So we know that $\{\tilde{V}^{\lambda}(y)\}$ and $\{\nabla\tilde{V}^{\lambda}(y)\}$ are both equicontinuous in $B_{R_1}(0)$. Since $\zeta^{\lambda}\rightharpoonup \zeta^*$ weakly star in $L^{\infty}(B_{R_1}(0))$, we have $\tilde{V}^{\lambda}(y)\rightarrow V^*(y)$ and $\nabla\tilde{V}^{\lambda}(y)\rightarrow \nabla V^*(y)$ a.e. in $B_{R_1}(0)$. By Arzela-Ascoli's theorem,
\begin{equation}\label{323}
\tilde{V}^{\lambda}\rightarrow V^*\ \ \ \ \text{in}\ \ C^1(\overline{B_{R_1}(0)}),
\end{equation}
as $\lambda\rightarrow+\infty.$
On the other hand, by a simple calculation we know that $V^{\lambda}$ satisfies,
\begin{equation}\label{322}
\begin{cases}
-\Delta V^{\lambda}= \zeta^{\lambda}-\frac{2\Omega}{\lambda}&\text{in } D_{\lambda}, \\
V^{\lambda}=\pi\left(\frac{\Omega}{2}-\mu^{\lambda}\right)&\text{on } \partial D_{\lambda}.
\end{cases}
\end{equation}
Recall that we have obtained the following estimate for $\mu^\lambda$ in Lemma \ref{811}
\[\mu^{\lambda}=\frac{1}{2\pi}\ln\frac{1}{\varepsilon}+O(1),\,\,\text{as }\lambda\rightarrow +\infty.\]
Set $d=\frac{1}{2}dist(X^*,\partial D)>0$. Then for sufficiently large $\lambda$, $d<dist(X^{\lambda}, \partial D)$, which implies ${d}/{\varepsilon}\leq |y| \leq 2/{\varepsilon}$ for any $y\in \partial D_{\lambda}$. Therefore for any $y\in \partial D_{\lambda}$
\[V^{\lambda}(y)=\pi(\frac{\Omega}{2}-\mu^{\lambda})=-\frac{1}{2}\ln\frac{1}{\varepsilon}+O(1)=-\frac{1}{2}\ln|y|+O(1).\]
Meanwhile, for any $y\in \partial D_{\lambda}$,
\begin{equation}
\begin{split}
\tilde{V}^{\lambda}(y)=&\frac{1}{2\pi}\int_{\mathbb{R}^2}\ln|y-y'|^{-1}\zeta^{\lambda}(y')dy' \\
=&\frac{1}{2\pi}\int_{B_{R_0}(0)}\ln|y|^{-1}\zeta^{\lambda}(y')dy'+O(1) \\
=&\frac{1}{2}\ln|y|^{-1}+O(1).
\end{split}
\end{equation}
Hence $|V^{\lambda}-\tilde{V}^{\lambda}|\leq C$ on $\partial D_{\lambda}$. Therefore, according to \eqref{322}, we get
\begin{equation}
\begin{cases}
&\Delta(V^{\lambda}-\tilde{V}^{\lambda})=\frac{2\Omega}{\lambda}\ \ \ \ \text{in}\ \ D_{\lambda}, \\
&|V^{\lambda}-\tilde{V}^{\lambda}|\leq C \ \ \ \ \ \ \ \ \ \ \ \ \ \text{on}\ \ \partial D_{\lambda},
\end{cases}
\end{equation}
which implies
\begin{equation}
\begin{cases}
&\Delta\left(V^{\lambda}-\tilde{V}^{\lambda}-\frac{\pi\Omega}{2}|X^{\lambda}+\varepsilon y|^2\right)=0\ \ \ \ \text{in}\ \ D_{\lambda}, \\
&\left|V^{\lambda}-\tilde{V}^{\lambda}-\frac{\pi\Omega}{2}|X^{\lambda}+\varepsilon y|^2\right|\leq C \ \ \ \ \ \ \ \ \text{on}\ \ \partial D_{\lambda},
\end{cases}
\end{equation}
where we used $\frac{\pi\Omega}{2}|X^{\lambda}+\varepsilon y|^2\leq C$ for any $y\in \partial D_{\lambda}$.
By the interior gradient estimate for harmonic functions, we deduce that
\[\sup_{y\in \overline{{B}_{{R_1}(0)}}}\left|\nabla V^{\lambda}(y)-\nabla\tilde{V}^{\lambda}(y)-\nabla\left(\frac{\pi\Omega}{2}|X^{\lambda}+\varepsilon y|^2\right)\right|\leq C\varepsilon.\]
Note that $|\nabla\left(\frac{\pi\Omega}{2}|X^{\lambda}+\varepsilon y|^2\right)|\leq C\varepsilon$ for any $y\in \overline{{B}_{{R_1}(0)}}$,~so we obtain
\[\sup_{y\in \overline{{B}_{{R_1}(0)}}}\left|\nabla V^{\lambda}(y)-\nabla\tilde{V}^{\lambda}(y)\right|\leq C\varepsilon.\]
Arzela-Ascoli's theorem yields that\,(up to a subsequence)\, there exists some constant $C^*$ such that
\[V^{\lambda}-\tilde{V}^{\lambda}\rightarrow C^*\ \ \\ \ \text{in}\ \ C^1(\overline{{B}_{{R_1}(0)}}),\]
which together with \eqref{323} gives
\begin{equation*}\label{324}
V^{\lambda}\rightarrow V^*+C^*\ \ \\ \ \text{in}\ \ C^1(\overline{{B}_{{R_1}(0)}}).
\end{equation*}
  Recall that if $\lambda>2\Omega$, then $\lambda|\{x\in D \mid \psi^{\lambda}(x)>0\}|= 1$, which yields $|\{y\in B_{R_1}(0)\mid V^{\lambda}(y)>0\}|= \pi$, and thus $C^*=0$. We note that this conclusion is independent of the choice of any subsequence, thus the proof is completed.

\end{proof}

Now we are ready to prove Theorem \ref{Thm}.
\begin{proof}[Proof of Theorem \ref{Thm}]
It suffices to show that any maximizer $w^\lambda$ satisfies \eqref{001}. For any $\phi\in C^{\infty}_c(D)$ and $x\in D$, consider the following ordinary differential equation
\begin{equation}\label{ORD}
\begin{cases}\frac{d\Phi_s(x)}{ds}=\nabla^\perp\phi(\Phi_s(x)) &s\in\mathbb R, \\
\Phi_0(x)=x.
\end{cases}
\end{equation}
Since $\nabla^\perp\phi\in C_c^\infty(D;\mathbb R^3)$, $\eqref{ORD}$ has a unique global solution. It is easy to check that $\nabla^\perp\phi$ is divergence-free, so $\Phi_s$ is a measure-preserving transformation from $D$ to $D$, that is, for any measurable set $F\subset D$, we have $|\{\Phi_s(x)\mid x\in F\}|=|F|$. Now we define a family of test functions $\{w_s\}_{t\in\mathbb R}$ by setting
\begin{equation}
w_s(x):=w^\lambda(\Phi_s(x)).
\end{equation}
It is easy to see that $w_s\in K_{\lambda}(D)$.  So we obtain
\begin{equation}\label{0000}
\frac{d\mathcal{E}(w_s)}{ds}|_{s=0}=0.
\end{equation}
Expanding $\mathcal{E}(\omega_s)$ at $s=0$, we have
\begin{equation}\label{0002}
\begin{split}
\mathcal{E}(w_s)=&\frac{1}{2}\int_D\int_DG(x,y)w^\lambda(\Phi_s(x))w^\lambda(\Phi_s(y))dxdy+\frac{\Omega}{2}\int_D|x|^2w^\lambda(\Phi_s(x))dx\\
=&\frac{1}{2}\int_D\int_DG(\Phi_{-s}(x),\Phi_{-s}(y))w^\lambda(x)w^\lambda(y)dxdy+\frac{\Omega}{2}\int_D|(\Phi_{-s}(x))|^2w^\lambda(x)dx\\
=&\mathcal{E}(w^\lambda)+s\int_Dw^\lambda(x)\nabla^\perp(Gw^\lambda(x)+\frac{\Omega}{2}|x|^2)\cdot\nabla\phi(x) dx+o(s).
\end{split}
\end{equation}
From \eqref{0000} and \eqref{0002} we get the desired result.

\end{proof}

\section{Proof of Theorem \ref{os}}
In this section we prove Theorem \ref{os}. To this end, we need two lemmas first.
\begin{lemma}\label{com}
Let $\{w_n\}$ be a maximizing sequence for $\mathcal{E}$ in $K_\lambda(D)$, then up to a subsequence there exists $w^\lambda\in \mathcal{S}_\lambda$ such that as $n\rightarrow+\infty$, $w_n\rightarrow w^\lambda$ in $L^p(D)$ for any $p\in[1,+\infty)$.
\end{lemma}
\begin{proof}
Since $\{w_n\}$ is a bounded sequence in $L^\infty(D)$, it suffices to show that $w_n\rightarrow w^\lambda$ in $L^2(D)$.
Firstly according to the proof of Proposition \ref{802}, for any maximizing sequence $w_n$, there must be a maximizer $w^\lambda\in K_\lambda(D)$ such that $w_n\rightarrow w^\lambda$ weakly star in $L^\infty(D)$. Thus $w_n\rightarrow w^\lambda$ weakly in $L^2(D)$, which implies
\begin{equation}\label{4-1}
\|w^\lambda\|_{L^2(D)}\leq\liminf_{n\rightarrow+\infty}\|w_n\|_{L^2(D)}.
\end{equation}
On the other hand, by Corollary \ref{coro}, $w^\lambda$ must have the form $w^\lambda=\lambda I_{\tilde{A}^\lambda}$ with $\lambda|\tilde{A}^\lambda|=1$, which gives
\begin{equation}\label{4-2}
\|w^\lambda\|_{L^2(D)}=\lambda|\tilde{A}^\lambda|^{1/2}=\lambda^{1/2}.
\end{equation}
But for any $n$,
\begin{equation}\label{4-3}
\|w_n\|_{L^2(D)}=\left(\int_D|w_n(x)|^2dx\right)^{1/2}\leq\lambda^{1/2}\left(\int_D|w_n(x)|dx\right)^{1/2}=\lambda^{1/2}.
\end{equation}
Combining \eqref{4-2} and \eqref{4-3} we obtain
\begin{equation}\label{4-4}
\|w^\lambda\|_{L^2(D)}\geq\limsup_{n\rightarrow+\infty}\|w_n\|_{L^2(D)}.
\end{equation}
Now by \eqref{4-1} and \eqref{4-4} we have
\begin{equation}\label{4-5}
\|w^\lambda\|_{L^2(D)}=\lim_{n\rightarrow+\infty}\|w_n\|_{L^2(D)}.
\end{equation}
By the uniform convexity of $L^2(D)$ we conclude that $w_n\rightarrow w^\lambda$ in $L^2(D)$.

\end{proof}

\begin{lemma}\label{dl}[\cite{B5}, Lemma 11, Lemma 12]
Let $w(x,t)\in L^\infty_{loc}(\mathbb R; L^p(D))$ with $3/4<p<+\infty$. Let $\mathbf{u}=\nabla^\perp Gw$, $\zeta_0\in L^p(D)$. Then there exists a unique weak solution $\zeta(x,t)\in L^\infty_{loc}(\mathbb R; L^p(D))\cap C(\mathbb R; L^p(D))$ to the following linear transport equation
\begin{equation}\label{lt}
\begin{cases}
\partial_t\zeta+\mathbf{u}\cdot\nabla\zeta=0, &t\in\mathbb R,\\
\zeta(\cdot,0)=\zeta_0.
\end{cases}
\end{equation}
Here by weak solution we mean
\begin{equation}
\begin{split}
\int_{\mathbb R}\int_D\partial_t\phi(x,t)\zeta(x,t)+\zeta(x,t)&(\mathbf{u}\cdot\nabla\phi)(x,t)dxdt=0,\,\,\forall\,\,\phi\in C_c^\infty(D\times\mathbb R),\\
&\lim_{t\rightarrow0}\|\zeta(\cdot,t)-\zeta_0\|_{L^p(D)}=0.
\end{split}
\end{equation}
Moreover, we have for any $t\in \mathbb R$
\begin{equation}\label{re}
|\{x\in D\mid \zeta(x,t)>a\}|=|\{x\in D\mid\zeta_0(x)>a\}|,\,\,\forall \,\,a\in\mathbb R.
\end{equation}
As a consequence, we have for any $t\in \mathbb R$
\begin{equation}
\|\zeta(\cdot,t)\|_{L^p(D)}=\|\zeta_0\|_{L^p(D)}.
\end{equation}

\end{lemma}

Now we are ready to prove Theorem \ref{os}.
\begin{proof}[Proof of Theorem \ref{os}]
We give the proof by contradiction. Suppose that there exists a $\delta_0>0$, $t_n>0,$ $v_0^n\in L^p(D)$ satisfying $dist_p(v_0^n,\mathcal{S}_\lambda)\rightarrow0,$ but
\begin{equation}\label{444}
dist_p(v^n_{t_n},\mathcal{S}_\lambda)>\delta_0.
\end{equation} Here $v^n_{t}$ is a weak solution to the vorticity equation with initial $v^n_0.$
Since $3/2\leq p<+\infty$, by energy and angular momentum conservation in Theorem {A} it is easy to check that $\{v^n_{t_n}\}$ satisfies
\begin{equation}\label{4-6}
\lim_{n\rightarrow+\infty}\mathcal{E}(v^n_{t_n})=\sup_{K_\lambda(D)}\mathcal{E}.
\end{equation}

\textbf{Case 1:} If $v_0^n\in K_\lambda(D)$, then the proof is easy. In this case, we have $v^n_{t_n}\in K_\lambda(D)$, so by Lemma \ref{com}, up to a subsequence $v^n_{t_n}\rightarrow w^\lambda$ in $L^p(D)$ for some $w^\lambda\in\mathcal{S}_\lambda$, which contradicts \eqref{444}.

\textbf{Case 2:} For general $v^n_{0}\in L^p(D)$, we need Lemma \ref{dl}.
Since $dist_p(v^n_0,\mathcal{S}_\lambda)\rightarrow0$, we can choose $w^n_0\in\mathcal{S}_\lambda$ such that as $n\rightarrow+\infty$
\begin{equation}\label{4-7}
\|w^n_0-v^n_0\|_{L^p(D)}\rightarrow0.
\end{equation}
Now for each $n$, let $w^n(x,t)$ be the solution of the following linear transport equation
\begin{equation}\label{4-8}
\begin{cases}
\partial_tw^n(x,t)+\nabla^\perp Gv^n_t\cdot\nabla w^n(x,t)=0,\\
w^n(x,0)=w^n_0(x).
\end{cases}
\end{equation}
By Lemma \ref{dl}, it is clear that $w^n(\cdot,t)\in K_\lambda(D)$ for any $t>0$, and as $n\rightarrow+\infty$
\begin{equation}\label{4-9}
\|w^n(\cdot,t_n)-v^n_{t_n}\|_{L^p(D)}=\|w^n_0-v^n_0\|_{L^p(D)}\rightarrow0.
\end{equation}
Combining \eqref{4-6} and \eqref{4-9} we obtain
\begin{equation}\label{4-10}
\lim_{n\rightarrow+\infty}\mathcal{E}(w^n(\cdot,t_n))=\sup_{K_\lambda(D)}\mathcal{E}.
\end{equation}
Then by Lemma \ref{com} we conclude that there exists $w^\lambda\in \mathcal{S}_\lambda$ such that $  \|w^n(\cdot,t_n)-w^\lambda\|_{L^p(D)}\rightarrow0$, which gives
\begin{equation}\label{4-11}
   dist_p(w^n(\cdot,t_n),\mathcal{S}_\lambda)\rightarrow0.
\end{equation}
Now \eqref{444},\eqref{4-9} and \eqref{4-11} together lead to a contradiction. Thus Theorem \ref{os} is proved.
\end{proof}

\noindent{\bf Acknowledgments:}
 Daomin Cao was partially supported by Hua Luo Geng Center of Mathematics,
AMSS, CAS, he was also partially supported by NNSF of China grant No.11331010 and No.11771469.
Guodong Wang was supported by NNSF of China grant No.11771469.

\end{document}